\documentclass{amsproc}
\usepackage{amsfonts}
\usepackage{amssymb}
\usepackage{eurosym}

\newtheorem{theorem}{Theorem}[section]

\newtheorem{prop}[theorem]{Proposition}
\newtheorem{corollary}[theorem]{Corollary}
\theoremstyle{definition}
\newtheorem{definition}[theorem]{Definition}
\newtheorem{example}[theorem]{Example}

\newtheorem{problem}[theorem]{Problem}
\theoremstyle{remark}
\newtheorem{remark}[theorem]{Remark}
\numberwithin{equation}{section}

\input{tcilatex}

\begin{document}
\title{Polynomials on Parabolic Manifolds}
\author{Ayd\i n Aytuna}
\address{FENS. Sabanc\i\ University, 34956, Tuzla, \.{I}stanbul, Turkey.}
\email{aytuna@sabanciuniv.edu}
\author{Azimbay Sadullaev}
\address{Mathematics Department, National University of Uzbakistan, VUZ
GORODOK, 700174, Taskent, Uzbakistan}
\email{sadullaev@mail.ru}
\maketitle

\begin{abstract}
A Stein manifold $X$ is called $S-parabolic$ if it possesses a
plurisubharmonic exhaustion function $\rho $ that is maximal outside a
compact subset of $X.$ In analogy with $(\mathbb{C}^{n},\ln |z|)$, one
defines the space of polynomials on a $S$-parabolic manifold $(X,\rho )$ as
the set of all analytic functions with polynomial growth with respect to $%
\rho $. In this work, which is, in a sense continuation of \cite{AS2}, we
will primarily study polynomials on $S$-parabolic Stein manifolds. In
Section \ref{sec1}, we review different notions of paraboliticity for Stein
manifolds, look at some examples and go over the connections between
parabolicity of a Stein manifold $X$ and certain linear topological
properties of the Fr\'{e}chet space of global analytic functions on $X$. In
Section \ref{sec2} we consider Lelong classes, associated Green functions
and introduce the class of polynomials in $S$-parabolic manifolds. In
Section \ref{sec3} we construct an example of a $S$-parabolic manifold, with
no nontrivial polynomials. This example leads us to divide $S$-parabolic
manifolds into two groups as the ones whose class of polynomials is dense in
the corresponding space of analytic functions and the ones whose class of
polynomials is not so rich. In this way we introduce a new notion of
regularity for $S$-parabolic manifolds. In the final section we investigate
linear topological properties of regular $S$-parabolic Stein manifolds and
show in particular that the space of analytic functions on such manifolds
have a basis consisting of polynomials. We also give a criterion for closed
submanifolds of a regular $S$-parabolic to be regular $S$-parabolic, in
terms of existence of tame extension operators for the spaces of analytic
functions defined on these submanifolds.
\end{abstract}








%
%
%


\section{Introduction}

\label{sec0} In the classical theory of Riemann surfaces one calls a Riemann
surface parabolic, in case every bounded (from above) subharmonic $(sh)$
function on $X$ reduces to a constant. Several authors introduced analogs of
these notions for general complex manifolds of arbitrary dimension in
different ways; in terms of triviality (parabolic type) and non-triviality
(hyperbolic type) of the Kobayashi or Caratheodory metrics, in terms of
plurisubharmonic $(psh)$ functions, etc. In this paper we will follow the
one dimensional tradition and call a complex manifold parabolic in case
every bounded from above plurisubharmonic function on it reduces to a
constant.

On the other hand, Stoll, Griffiths, King, et al. in their work on
Nevanlinna's value distribution theory in higher dimensions, introduced
notions of \textit{"parabolicity"} in several complex variables by requiring
the existence of \textit{special plurisubharmonic ($psh$) exhaustion
functions}. Following Stoll \cite{St2}, we will call an $n$-dimensional
complex manifold $X$, \textit{$S$-parabolic} in case there is a
plurisubharmonic function $\rho$ on $X$ with the properties:

\begin{itemize}
\item[a)] $\lbrace z\in X:\rho(z)<C\rbrace \subset \subset X, \, \forall
C\in \mathbb{R}$ (i.e. $\rho$ is exhaustive),

\item[b)] the Monge - Amp\`{e}re operator $(dd^c\rho)^n$ is zero off a
compact $K\subset\subset X$ . That is $\rho$ is maximal plurisubharmonic
function outside $K$.
\end{itemize}

If in addition we can choose $\rho$ to be continuous then we will say that $%
X $ is $S^*$-parabolic. Special exhaustion functions with certain regularity
properties play a key role in the Nevanlinna's value distribution theory of
holomorphic maps $f:X\to P^m$, where $P^m$ is the $m$-dimensional projective
manifold (see. \cite{GK},\cite{S2}, \cite{Sh}, \cite{St1}, \cite{St2}).

On the other hand, for manifolds which have a special exhaustion function
one can define extremal Green functions as in the classical case and apply
pluripotential theory techniques to obtain analogs of some classical results
(see \cite{S6}, \cite{Ze1}, \cite{Ze2}) and section \ref{sec2} below.

Most of the previous papers on the subject required additional smoothness
conditions for the special exhaustion functions. Note that we only
distinguish the cases when the special exhaustion function is continuous or
just plurisubharmonic. Also note that without the maximality condition b),
an exhaustion function $\rho(z)\in C(X)\cap psh(X)$ always exists for any
Stein manifold $X$. This follows from the fact that any Stein manifold $X$, $%
dim(X)=n$, can be properly embedded in $\mathbb{C}^{2n+1}_W$, hence one can
take for $\rho$ the restriction of $\ln|w|$ to $X$.

In this paper, we will primarily study polynomials in $S$-parabolic Stein
manifolds. Polynomials in $S^*$-parabolic manifolds were introduced by A.
Zeriahi in \cite{Ze1}. However, his investigations were more focused on
polynomials on affine algebraic varieties. In analogy with $(\mathbb{C}^n,
\ln|z|)$ one defines polynomials in a $S$-parabolic manifold $(X, \rho)$ as
the set of all analytic functions with polynomial growth with respect to $%
\rho$.

The organization of the paper is as follows: In \ref{sec1}, we review
different notions of paraboliticity for Stein manifolds, look at some
examples and go over the connections between parabolicity of a Stein
manifold $X$ and certain linear topological properties of the Fr\'{e}chet
space of global analytic functions on $X$. In \ref{sec2} we consider Lelong
classes, associated Green functions and introduce the class of polynomials
in $S$-parabolic manifolds. These two sections are written in a survey
style. In \ref{sec3} we construct an example of a $S^*$-manifold, with no
nontrivial polynomials. This example leads us to divide $S$-parabolic
manifolds into two groups as the ones whose class of polynomials is dense in
the corresponding space of analytic functions and the ones whose class of
polynomials is not so rich. In this way we introduce a new notion of
regularity for $S$-parabolic manifolds. In the final section we investigate
linear topological properties of regular $S$-parabolic Stein manifolds and
show in particular that the space of analytic functions on such manifolds
have a basis consisting of polynomials. In this section we also give a
criterion for closed submanifolds of a regular $S^*$-parabolic to be regular 
$S^*$-parabolic, in terms of existence of tame extension operators for the
spaces of analytic functions defined on these submanifolds.


\section{Parabolic manifolds}

\label{sec1} In this section we will review notions of parabolicity for
Stein manifolds, look at some examples and go over the relation between
parabolicity of a Stein manifold $X$ and certain linear topological
properties of the Fr\'echet space of global analytic functions on $X$.

\begin{definition}
\label{page1} A Stein manifold $X$ is called parabolic, in case it does not
possess a non-constant bounded above plurisubharmonic function.
\end{definition}

Thus, parabolicity of $X$ is equivalent to the following: if $u (z)\in psh
(X)$ and $u (z) < C$, then $u (z) \equiv const$ on $X$.

It is very convenient to describe parabolicity in term of well-known $%
\mathcal{P}$-measures of pluripotential theory \cite{AS2}, \cite{S3}. Let
our Stein manifold $X$ be properly imbedded in $\mathbb{C}^{2n+1}_w,n = \dim
X$, and denote by $\sigma (z)$ the restriction of $\ln |w|$ to $X$. Then $%
\sigma (z) \in C(X)\cap psh (X)$ and $\{\sigma (z) < C\}\subset\subset X,
\forall C \in R$. We assume $0$ is not in $X$ and that, min $\sigma (z) < 0$%
. We consider $\sigma$-balls $B_R = \{z \in X : \sigma (z) < \ln R\}$ and as
usual, define the class $\mathcal{U}(\overline{B}_1,B_R)= \{u\in psh(B_R) :
u|_{\overline{B}_1}<-1, \,\,u|_{B_R}<0 \}$. Then the function 
\begin{equation*}
\omega (z,\overline{B}_1,B_R)=\sup \{u(z):u\in\mathcal{U}(\overline{B}%
_1,B_R) \}
\end{equation*}
is called as $\mathcal{P}$-measure of the $\overline{B}_1$ with respect to
the domain $B_R$. $\mathcal{P}$-measure $\omega (z,\overline{B}_1,B_R)$ is
plurisubharmonic in $B_R,$ is equal to $-1$ on $\overline{B}_1$ and tends to 
$0$ for $z\rightarrow \partial B_R$. Moreover, it is maximal, i.e. $%
(dd^c\omega)^n=$ in $B_R\setminus\overline{B}_1$. Since $\omega (z,\overline{%
B}_1,B_R)$ decreases with $R \nearrow \infty$, and the limiting function
satisfies: 
\begin{equation*}
\omega (z,\overline{B}_1)=\lim_{R\rightarrow\infty}\omega (z,\overline{B}%
_1,B_R)\in psh(X),\,\, \omega (z,\overline{B}_1)|_{\overline{B}_1}\equiv
-1,\,\, \omega (z,\overline{B}_1)<0 \, \,\forall z \in X.
\end{equation*}
The proposition below, while not difficult to prove, is sometimes very
useful.

\begin{prop}
The Stein manifold $X$ is parabolic if and only if $\,\omega (z,\overline{B}%
_1) \,$ is trivial, i.e. $\,\omega (z,\overline{B}_1)\equiv -1\,.$
\end{prop}

We note, that triviality of $\omega (z,\overline{B}_1)$ does not depend upon 
$\overline{B}_1$; one can take instead of $\overline{B}_1$, any closed ball $%
\overline{B}_r$ or even, any pluriregular compact set $E\subset X$ (see \cite%
{AS2}).

\begin{definition}
A Stein manifold $X$ is called $S$-parabolic, if there exit exhaustion
function $\rho (z)\in psh (X)$ that is maximal outside a compact subset of $%
X $. If in addition we can choose $\rho (z)$ to be continuous then we will
say that $X$ is $S^*$-parabolic.
\end{definition}

A plurisubharmonic exhaustion function that is maximal outside a compact
subset will be referred to as special plurisubharmonic exhaustion. We will
tacitly assume, unless otherwise stated that special exhaustion functions
are maximal on the sets where they are strictly positive. It is not
difficult to see that $S$-parabolic manifolds are parabolic.

In fact, since the special exhaustion function $\rho (z)$ of a $S$-parabolic
manifold $(X, \rho )$ is maximal off some compact $K\subset\subset X$ we can
choose a positive $r$, so that $B_r=\{\rho (z)<\ln r\}$ contains $K$. For $%
R>r$ the $\mathcal{P}$-measure can be calculated as; 
\begin{equation*}
\omega (z,\overline{B}_1, B_R)=\left\{-1,\frac{\rho (z)-R}{R-r} \right\},
\end{equation*}
From here it follows, that $\lim_{R\rightarrow \infty}\omega (z,\overline{B}%
_1, B_R)\equiv -1.$

For open Riemann surfaces the notions of $S$-parabolicity, $S^*$%
-parabolicity and parabolicity coincide. This is a consequence of the
existence of Evans-Selberg potentials (subharmonic exhaustion functions that
are harmonic outside a given point) on parabolic Riemann surfaces \cite{SN}.
Authors do not know any prove of the following important problems in the
multidimensional case $n=\dim X >1$.

\begin{problem}
\label{problem1} Do the notions of $S$-parabolicity and $S^*$-parabolicity
coincides for the Stein manifolds of arbitrary dimension?
\end{problem}

\begin{problem}
\label{problem2} Do the notions of parabolicity and $S$-parabolicity
coincide for the Stein manifolds of arbitrary dimension?
\end{problem}

The prime example of an $S^*$-parabolic manifold is of course $\mathbb{C}^n$%
, with the special exhaustion function $\ln^+|z|$. Algebraic affine
manifolds, with their canonical special exhaustion functions as described in 
\cite{GK} also forms an important class of $S^*$-parabolic manifolds.

Another set of indicative examples could be obtained by considering closed
pluripolar subsets in $\mathbb{C}^n$ , whose complements are pseudoconvex.
Such sets are called \textit{``analytic multifunctions"} by some authors.
They are studied extensively and are extremely important in approximation
theory, in the theory of analytic continuation and in the description of
polynomial convex hulls (see \cite{AW,BR,N,O,S7,Sl1,Sl2} and others). It is
clear, that these sets are removable for the class of bounded
plurisubharmonic functions defined on their complements. Hence their
complements are parabolic Stein manifolds. We would like to state the
following special case of Problem \ref{problem2} above, with the hope that
it will be more tractable:

\begin{problem}
\label{problem3} Let $A$ be an analytic multifunction in $\mathbb{C}^n$. Is $%
X=\mathbb{C}^n\setminus A$, $S$-parabolic?
\end{problem}

In classical case, $n = 1$, every closed polar set $A\subset \mathbb{C}$ is
an analytic multifunction. As was remarked above, in this case $\mathbb{C}%
^n\setminus A$ is $S^*$-parabolic. On the other hand if $A= \{
p(z)=0\}\subset\mathbb{C}^n$ is an algebraic set, where $p$ is a polynomial,
assuming that $0\notin A$, it is not difficult to see that the function 
\begin{equation*}
\rho (z)=-\frac{1}{\deg p}\ln |p(z)|+2\ln |z|
\end{equation*}
gives rise to a special exhaustion function for $\mathbb{C}^n\setminus A$.
More generally we have:

\begin{theorem}[\protect\cite{AS2}]
\label{thm27} Let 
\begin{equation*}
A=\left \{z=(^{\prime }z,z_n)=(z_1,z_2,...,z_n)\in\mathbb{C}%
^n:\,\,z_n^k+f_1(^{\prime }z)z_n^{k-1}+...+f_k(^{\prime }z)=0 \right \}
\end{equation*}
be a Weierstrass polynomial (algebraiodal) set, where $f_j\in O(\mathbb{C}%
^{n-1})$ are entire functions, $j = 1,2,...,k,$ $k > 1$. Then $X=\mathbb{C}%
^n\setminus A$ is $S^*$-parabolic. Moreover, the function $\rho (z)=-\ln
|F(z)|+\ln\left(|z|^{\prime 2}+|F(z)-1|^2\right)$ is a plurisubharmonic
exhaustion function for $X,$ that is maximal outside a compact subset of $X$.
\end{theorem}

For more examples of parabolic manifolds we refer the reader to \cite{AS2}.

It turns out that the paraboliticity of a Stein manifold $X$ and certain
linear topological properties of the Fr\'{e}chet space of analytic functions
on $X$ are connected. We will end this section by reviewing some results
obtained in this context.

As usual, the topology on the space of analytic functions on a complex
manifold $X$, $O\left( X\right) $ is the topology of uniform convergence on
compact subsets of $X,$ which makes $O\left( X\right) $ a nuclear Fr\'{e}%
chet space. We start by recalling the $DN$ condition of Vogt from the
structure theory of Fr\'{e}chet spaces;

\begin{definition}
A Fr\'{e}chet space $Y$ has the \textit{property $DN$} in case for a system $%
\left( \left\Vert {\cdot}\right\Vert _{k}\right) $ of seminorms generating
the topology of $Y$ one has:%
\begin{equation*}
\exists \text{ }k_{0}\text{ such that }\forall p\text{ }\exists q\text{ ,}%
\,\,C>0:\,\,\left\Vert x\right\Vert _{p}\leq C\left\Vert x\right\Vert
_{k_{0}}^{\frac{1}{2}}\left\Vert x\right\Vert _{q}^{\frac{1}{2}}\text{ ,}%
\forall x\epsilon Y
\end{equation*}
\end{definition}

This condition does not depend on the choice of generating seminorms. For
this and related linear topological invariants we refer the reader to the
book (\cite{VM}).

The first result we will state is an adaptation of a result from (\cite{A4})
part of which were proved by D.Vogt, V.Zaharyuta, and the first author
independently.

\begin{theorem}
\label{thm29} For a Stein manifold $X$of dimension $n$, the following
conditions are equivalent:

1. $X$ is parabolic

2. $O\left( X\right) $ has the property $DN$

3. $O\left( X\right) $ is isomorphic as Fr\'{e}chet spaces to $O\left(
C^{n}\right) .$
\end{theorem}

\bigskip

Mitiagin and Henkin, in their seminal paper ( \cite{MH}) initiated a program
which they called "linearization of the basic theorems of complex analysis".
One of the problems they considered (in connection with Remmert's theorem)
was the possibility of finding continuous linear right inverse operators to
the restriction operator for analytic functions defined on closed complex
submanifolds of $\mathbb{C}^{N}.$ In other words for a closed complex
submanifold $V$ of some $\mathbb{C}^{N},\,$denoting by $R$ the restriction
operator from $O\left( \mathbb{C}^{N}\right) $ onto $O\left( V\right) $ the
query was to find a continuous linear (extension) operator $E:O\left(
V\right) \rightarrow O\left( \mathbb{C}^{N}\right) $ such that $R\circ
E=Identity$ on $O\left( V\right) .$ Mitiagin and Henkin stated ( Proposition
6.5 \cite{MH}) that this was possible in case $O\left( V\right) $ is
isomorphic to $O\left( \mathbb{C}^{n}\right) , $ $n=\dim V.$ A complete
answer to this query was given by Vogt \cite{V1} (see also \cite{V2}, \cite%
{V3}), which in our terminology reads as follows:

\begin{theorem}
\label{thm210} A Stein manifold is parabolic if and only if whenever it is
embedded into a Stein manifold as a closed submanifold, it admits a
continuous linear extension operator.
\end{theorem}

\bigskip

We now wish to pass to a more refined category of Fr\'{e}chet spaces. Recall
that a graded Fr\'{e}chet space is a tuple $\left( Y,{\left\Vert {\cdot}%
\right\Vert _{s}}\right), $ where $Y$ is a Fr\'{e}chet space and $\left(
\left\Vert {\cdot}\right\Vert _{s}\right) $ is a fixed system of seminorms
on $Y$ defining the topology. The morphisms in this category are tame linear
operators.

\begin{definition}
A continuous linear operator $T$ between two graded Fr\'{e}chet spaces $%
\left( Y,{\left\Vert {\cdot}\right\Vert _{s}}\right) $ and $\left( Z,{%
\left\vert {\cdot}\right\vert _{k}}\right) $ is said to be \textit{tame} in
case:%
\begin{equation*}
\exists \text{ }A>0\text{ }\forall k\text{ }\exists \text{ }C>0:\left\vert
T\left( x\right) \right\vert _{k}\leq C\left\Vert x\right\Vert _{k+A}.
\end{equation*}
\end{definition}

Two graded Fr\'{e}chet spaces are called tamely isomorphic in case there is
a one to one tame linear operator from one onto the other whose inverse is
also tame.

On a Stein manifold $X$, each exhaustion $\left( K_{s}\right) _{s=1}^{\infty
}$ of holomorphically convex compact sets with $K_{s}\subset \subset int
K_{s+1},\,\,$ $s=1,2,..$, induces a grading $\left\{ \left\Vert {\cdot}%
\right\Vert _{K_{s}}\right\} $ on $O\left( X\right) $ by considering the sup
norms on these compact sets.

\begin{theorem}
(\cite{AS2}) A Stein manifold of dimension $n$ is ${S}^{*}$-parabolic if and
only if there exits an exhaustion $\,\left( K_{s}\right) _{s=1}^{\infty }\,\,
$of $X\,$ such that the graded spaces $\left( O\left( X\right) ,{\left\Vert {%
\cdot}\right\Vert _{K_{s}}}\right) $ and $\left( O\left( \mathbb{C}%
^{n}\right) ,{\left\Vert {\cdot}\right\Vert _{P_{s}}}\right) $ are tamely
isomorphic, where $P_{s}= \left( z\in \mathbb{C}^{n}:\,\,\left\Vert
z\right\Vert \leq e^{s}\right) $, \newline
$s=1,2,... $ .
\end{theorem}

\bigskip

This result displays the similarities between function theories on ${S}^{* }$%
-parabolic manifolds and the complex Euclidean spaces, however finding tame
isomorphisms to the space of entire functions may not be an easy task. On
the other hand graded Fr\'{e}chet spaces tamely isomorphic to \textit{%
infinite type power series spaces }were studied by various authors ( \text{
see for example, \cite{poppenbergt} }) and linear topological conditions
that ensure the existence of such isomorphisms were obtained. Recall that
for an exponent sequence $\alpha =\left( \alpha _{m}\right) _{m}$ ; $\alpha
_{m}\uparrow \infty ,$ the power series space of infinite type is the graded
Fr\'{e}chet space 
\begin{equation*}
\Lambda _{\infty }\left( \alpha \right) =\left\{ \mathbf{\xi =}\left( \xi
_{m}\right) _{m}:\left\vert \mathbf{\xi }\right\vert _{k}\doteq
\sum\limits_{m=1}^{\infty }\left\vert \xi _{m}\right\vert e^{k\alpha
_{m}}<\infty ,\text{ }k=1,2,...\right\}
\end{equation*}%
equipped with the grading $\left( \left\vert {\cdot}\right\vert _{k}\right)
_{k=1}^{\infty }.$

\begin{theorem}
( \cite{AS2}) A Stein manifold $X$of dimension $n$ is ${S}^{* }$-parabolic
in case there exits an exhaustion $\left( K_{s}\right) _{s=1}^{\infty }$of $X
$ such that $\left( O\left( X\right) ,{\left\Vert {\cdot}\right\Vert _{K_{s}}%
}\right) $ is tamely isomorphic to an infinite type power series space of
infinite type.
\end{theorem}

\bigskip

Given a ${S}^*$-parabolic Stein manifold $X,$ $\dim X=n,$ with a special
exhaustion function $\rho ,$ a natural grading for $O\left( X\right) $ can
be obtained by considering the grading induced by the exhaustion $\left( 
\overline{D}_{k} \right) _{k=1}^{\infty }$ where $D_{k}=\left( z:\,\rho
\left( z\right) <k\right),\,\,$ $k=1,2,...,\,\,$are the sub-level sets of $%
\rho.\,\,$ We will conclude this section with a result about the Fr\'{e}chet
space structure of this graded space.

\bigskip

\begin{theorem}
( \cite{AS2}) With the above notation the graded Fr\'{e}chet space $\left(
O\left( X\right) ,{\left\Vert {\cdot}\right\Vert _{\overline{D}_{s}}}\right) 
$ is tamely isomorphic to an infinite type power series space $\Lambda
_{\infty }\left( \alpha \right) $ where the sequence $\alpha =\left( \alpha
_{n}\right) _{n}$ satisfies 
\begin{equation*}
\lim_{m}\frac{\alpha _{m}}{m^{\frac{1}{n}}}=2\pi \left( n!\right) ^{\frac{1}{%
n}}\left( \int\limits_{X}\left( dd^{c}\rho \right) ^{n}\right) ^{-\frac{1}{n}%
}.
\end{equation*}
\end{theorem}


\section{Aspects of pluripotential theory on ${S}$-parabolic manifolds \label%
{sec2}}

The complex space $\mathbb{C}^n$ with the special exhaustion function $\log
|z|$ is a classical and inspiring example of a parabolic manifold. One can
introduce a pluripotential theory on a $S$-parabolic manifold $(X,\rho)$ by
taking the well-studied complex pluripotential theory on $\mathbb{C}^n$ as a
model and by using $\rho$ instead of $\log |z|$. On $S^*$-parabolic
manifolds, analogs of basic notions of classical pluripotential theory were
introduced by Zeriahi \cite{Ze1} (see also \cite{AS2}). In this section we
introduce the analog of classical Lelong classes for parabolic manifolds
with not-necessarily continuous special exhaustion functions i.e. for ${S}$%
-parabolic manifolds and consider certain plurisubharmonic functions
belonging to this class.

\begin{definition}
Let $(X,\rho)$ be a ${S}$-parabolic manifold. The class 
\begin{equation*}
\mathcal{L}_{\rho}=\{u(z)\in PSH(X): \,\, u(z)\leq \mathrm{c_u}+{\rho}^+(z)
\,\forall z \in X\},
\end{equation*}
where $c_u$ is a constant, ${\rho}^+(z)=\max\{ 0, p(z)\}$, will be called
the \emph{Lelong class} corresponding to the special exhaustion function ${%
\rho}$. By $\mathcal{L}_{\rho}(K),$ $K\subset X$ a compact set, we denote
the class 
\begin{equation*}
\mathcal{L}_{\rho}(K)=\{u\in \mathcal{L}_{\rho}: \,\,u|_K\leq 0 \}.
\end{equation*}
The analog of Zaharyuta-Siciak etremal function for this class i.e. the
upper regularization ${V}^*(z,K)=\lim V(z,K)$ of ${V}(z,K)=\sup \{u(z)\in%
\mathcal{L}(K)\}$ will be called the ${\rho}$-Green function of $K$.
\end{definition}

Note that ${V}^*(z,K)$ could either be identically $+\infty$ (if $K$ is
pluripolar) or it belongs to $\mathcal{L}_p$ and defines a special
exhaustion function for $X$ (if $K$ is not pluripolar).

\emph{Pluriregular points}, for a compact $K\subset X$, can be defined, in
accordance with the classical case, as the points $z_0\in X$ for which ${V}%
^*(z,K)=0$. A compact set $K\subset X$ will be called \emph{pluriregular} in
case all of its points are pluriregular i.e. ${V}^*(z,K)=0\,\forall z\in K$.
It is not difficult to show, arguing as in the classical case, that the
closure $\overline{D}$ of a domain $D\subset X$ with the piecewise smooth
boundary, $\partial D \in C^1$, is pluri-regular. Consequently there is a
rich supply of pluri-regular compact set for a given ${S}$-parabolic
manifold.

On the space $\mathbb{C}^n$ it is a classical fact due to Zaharyuta that for
a compact pluriregular set, ${V}(z ,K)$ is a continuous function (see \cite%
{K}). Zeriahi observed that the same result is valid for ${S}^*$-parabolic
manifolds \cite{Ze1}. On the other hand for a ${S}$-parabolic manifold $X$
if ${V}^*(z,K)\in C(X)$ for a compact $K\subset\subset X$, then $X$ becomes
a ${S}^*$-parabolic manifold. In fact in this case one can take ${V}^*(z,K)$
as a special exhaustion function for $X$.

Our next theorem gives a criterion for checking continuity of ${V}^*(z,K)$
for pluriregular compact subsets of a ${S}$-parabolic manifold $X$.

\begin{theorem}
\label{thm4} (see \cite{AS1}). Let $(X,\rho)$ be a ${S}^*$ -parabolic
manifold with special exhaustion function $\rho(z)\in psh(X)$ and let $%
\rho_*(z)=\varliminf_{\,w\to z}\rho(z)$ be the measure of discontinuity of $%
\rho$ at the point $z\in X$. If 
\begin{equation}  \label{eqn2}
\varlimsup_{\rho(z)\to \infty}\frac{\rho(z)}{\rho_*(z)}=\lim_{\rho(z)\to
\infty}\frac{\rho(z)}{\rho_*(z)}=1
\end{equation}
then $V^*(z,K)\in C(X)$ for any pluriregular compact $K\subset X$.
\end{theorem}

We note, that the condition \eqref{eqn2} means continuity of $\rho(z)$ at
infinitive points of $X$.

\begin{proof}
We fix a pluriregular compact $K\subset X$ and take the Green function $%
V^*(z,K)$. It is clear, that there exist a constants $C_1,C_2$: 
\begin{equation*}
C_1+\rho^+(z)\leq V^*(z,K)\leq C_2+\rho^+(z)\, \,\forall z\in X.
\end{equation*}
It follows, that the Green function $\nu(z)=V^*(z,K)$ also satisfies the
condition \eqref{eqn2}.

By the approximation theorem (see \cite{FN}, \cite{S4}) we can approximate $%
V^*(z,K)\in psh(X)$: we can find a sequence of smooth $psh$ functions 
\begin{equation*}
\nu_j(z)\in psh(X)\cap C^\infty(X),\, \nu_j(z)\downarrow \nu(z),\,z\in X.
\end{equation*}

Since $K\subset X$ is pluriregular, then $\nu|_K\equiv 0$ and for fixed $%
\varepsilon >0$ we take the neighborhood $U=\lbrace
\nu(z)<\varepsilon/2\rbrace\supset K$. Applying for $K\subset U$ the
well-known Hartog's lemma to $\nu_j(z)\downarrow \nu(z)$, we have: 
\begin{equation*}
\nu_j(z)<\varepsilon, \, \forall j\geq j_0, \, z\in K.
\end{equation*}

By \eqref{eqn2} there exists $R>0$ such that 
\begin{equation}  \label{eqn3}
\nu(z)<\nu_*(z)+\varepsilon \nu_*(z),\, z \notin B_R=\lbrace z\in X:
\nu(z)<R\rbrace ,\, B_R\supset K.
\end{equation}

If $z\in\partial B_R$, then by \eqref{eqn3}, $\nu(z)<(1+\varepsilon)
\nu_*(z)\leq (1+\varepsilon)R$. Applying again the Hartog's lemma we have 
\begin{equation*}
\nu_j(z)<(1+2\varepsilon)R,\, j>j_1\geq j_0,\, z\in \partial B_R.
\end{equation*}

Fix $j>j_1$ and put 
\begin{equation*}
w(z) = \left\{ 
\begin{array}{lc}
\max\lbrace \nu_j(z), (1+3\varepsilon)\nu(z)\rbrace & \text{if } z\in B_R,
\\ 
(1+3\varepsilon)\nu(z)-\varepsilon R & \text{if } z\notin B_R.%
\end{array}
\right.
\end{equation*}
Since for $z\in\partial B_R$ we have $w(z)=(1+3\varepsilon)\nu(z)-%
\varepsilon R\geq (1+3\varepsilon) R-\varepsilon R=(1+2\varepsilon) R\geq
\nu_j(z)$, then $w(z)\in psh(X)$. Hence, the function 
\begin{equation*}
\frac{1}{1+3\varepsilon}(w(z)-\varepsilon)\in \mathcal{L}.
\end{equation*}

Since for $z\in K$ this function is negative, then 
\begin{equation*}
\frac{1}{1+3\varepsilon}(w(z)-\varepsilon)\leq V^*(z,K).
\end{equation*}

It follows, that $\nu_j(z)\leq (1+3\varepsilon)V^*(z,K)+\varepsilon,\, z\in
B_R$. This with\newline
$\nu_j(z)\geq V^*(z,K)$ gives continuity of $V^*(z,K)$ in $B_R$ and
consequently on $X$.
\end{proof}

Note that Theorem \ref{thm4} follows, that in the condition \eqref{eqn2} $X$
is ${S}^*$-parabolic. On the other hand if $X$ is ${S}^*$-parabolic, i.e. $%
\rho$ is continuous, then $\rho(z)\equiv\rho_*(z)$, so the condition %
\eqref{eqn2} is satisfied automatically.

We will now introduce the main objects of our study, namely the polynomials
on ${S}$-parabolic manifolds.

\begin{definition}
Let $(X,\rho)$ be a ${S}$-parabolic manifold. A holomorphic function $f\in {O%
} (X)$ is called a \emph{polynomial} on $X$ in case for some integers $d$
and $c>0$ $f$ satisfies the growth estimate 
\begin{equation*}
\ln |f(z)|\leq d\cdot \rho^+(z)+c\quad \forall z\in X.
\end{equation*}
The minimal such $d$ will be called the \emph{degree} of $f$ and the set of
all polynomials on $X$ with degree less than or equal to $d$ will be denoted
by $\mathcal{P}_\rho^d.$

A. Zeriahi, using an idea of Plesniak \cite{P1} showed that the vector
spaces $\mathcal{P}_p^d$, for an ${S}$-parabolic manifold is finite
dimensional, and give bounds for their dimension \cite{Ze1}. We will give a
different proof of this result using techniques of \cite{AS1}.

\begin{theorem}
Let $(X,\rho)$ be an ${S}$-parabolic Stein manifold. The space $\mathcal{P}%
_\rho^d$ is a finite dimensional complex vector space and there exists a $%
C=C(X)>0$ such that $\dim \mathcal{P}_p^d \leq Cd$.
\end{theorem}

\begin{proof}[sketch of the proof]
Let us choose $\delta(d)$ linearly independent elements from $\mathcal{P}%
_\rho^d$. Fix a  pluriregular compact set $K$ and any domain $D$ with $%
K\subset D\subset\subset X$. We choose an $R_D\in\mathbb{N}$ such that $%
\overline{D}\subset \{z: {V}^*(z, K) < R_D\}$.

Any polynomial $p$ of degree less than or equal to $d$, satisfies 
\begin{equation*}
\frac{1}{d}\ln \left(\frac{|p(z)|}{||p||_K}\right)\leq {V}^*(z,K) \quad
\forall z\in X.
\end{equation*}
The norm's we will use in this proof are the sup norms.

In particular we have 
\begin{equation*}
||p||_{\overline{D}}\leq e^{d\cdot R_D}||p||_K \quad \forall p\in \mathcal{P}%
_p^d.
\end{equation*}
At this point we will put to use two results from functional analysis: the
first is the well-known theorem of Tichomirov which in our setting says that
the above estimate yields an estimate from below of the $\delta (d)-1$'th
Kolmogorov diameter in $C(K)$ of the restriction of the unit ball $O(X)|_{%
\overline{D}}\subset C(\overline{D})$ to $K$ and a general fact from \cite%
{A5} that says it is possible to choose a $D$ for this $K$ such that the
sequence of Kolmogorov diameters considered above is weakly asymptotic $%
\{e^{-m^{1/n}}\}$. We refer the reader to \cite{AKT1} for details. By
choosing $D$ suitable, one gets 
\begin{equation*}
\exists C_1 >0: \,\, e^{-(\delta (d)-1)^{1/n}} \geq C_1e^{d\cdot R_D}.
\end{equation*}
Hence 
\begin{equation*}
\exists C_2 >0:\,\, \delta (d)\leq C_2d^n \text{ for all }d=1,2,...\,.
\end{equation*}
\end{proof}

In the case of algebraic affine manifolds of dimension $n$ with canonical
special exhaustion function, we actually have that the sequence $\{\dim%
\mathcal{P}_\rho^d\}_d$ and $\{d^n\}_d$ are weakly asymptotic i.e. $\exists
C_1>0$ and $C_2>0$ such that 
\begin{equation*}
C_1 \leq {\underline{\lim}}\,_{d \to \infty } \,{\frac{\dim\mathcal{P}_p^d}{%
d^{1/n}}} \leq \overline{\lim}_{d \to \infty}\, {\ \frac{\dim \mathcal{P}_
\rho^d}{d^{1/n}}} \leq C_2.
\end{equation*}
For more information on these matters we refer to the reader to \cite{Ze3}
and \cite{AS1}.
\end{definition}


\section{Example}

\label{sec3} In this section we will construct a parabolic manifold for
which there are no non-trivial polynomials. In the first part of the section
we will first construct a compact polar set $K\subset\mathbb{C}$ and a
subharmonic function $u(z)$ on the complex plane $\mathbb{C}$, harmonic in $%
\mathbb{C} \setminus K$, for which $u|_K=-\infty$ and 
\begin{equation*}
\lim_{z\rightarrow K}\frac{u(z)}{\ln dist(z,K)}=0. 
\end{equation*}
The condition above means, in particular, that near $K$, the function $|u(z)|
$ is smaller than $\varepsilon |\ln dist(z,K)|$. We note that for compact
sets containing an isolated point, such that function does not exists.

In the second part of the section we will use this example to construct

\begin{theorem}
There exists a polar compact $K\subset\mathbb{C}$ and a subharmonic function 
$u(z)$ on the complex plane $\mathbb{C}$, harmonic in $\mathbb{C} \setminus K
$, for which $u|_K=-\infty$, and 
\begin{equation}
\lim_{z\rightarrow K}\frac{u(z)}{\ln dist(z,K)}=0.  \label{eq:4}
\end{equation}
\end{theorem}

\begin{proof}
We take a special Cantor set $K\in [0,1]\subset\mathbb{C}^n$ and the
probability measure $\mu$, $supp\mu\subset K $ on it such that, the
potential of $\mu$ tends to $-\infty$ slowly than any $\varepsilon\ln
dist(z,K)$ $\forall\varepsilon>0$.

Consider the segment $[0,1]$, and denote it as $K_0=[a_{01},b_{01}]$, the
length of $K_0$ is 1. Next we proceed as in the construction of Cantor sets:
fix $\delta =1/4$ and the sequence $t_m=4^{m-1}$, $m=1,2,...$ From $%
(a_{01},b_{01})$ we put off the interval $[a_{01}+\delta,b_{01}-\delta]$. We
get the union of two segments, $K_1=[a_{01},a_{01}+\delta]\cup
[a_{02}-\delta ,a_{02}]$. Redenote them as $K_1=[a_{01},a_{01}+\delta]\cup
[b_{01}-\delta ,b_{01}]=[a_{11},b_{11}]\cup [a_{12},b_{12}]$. Distances
between knot-points $a_{11},b_{11},a_{12},b_{12}$ are: 
\begin{equation*}
|b_{1j}-a_{1j}|=\delta,j=1,2, \quad |b_{11}-a_{12}|=1-2\delta.
\end{equation*}
Then with each of these segments we do the same procedure, changing $\delta $
to $\delta^{t_2}$: we get 4 segments, 
\begin{equation*}
\begin{array}{ll}
K_2 & =[a_{11},a_{11}+\delta^{t_2}]\cup [b_{11}-\delta^{t_2},b_{11}]\cup
[a_{12},a_{12}+\delta^{t_2}]\cup [b_{12}-\delta^{t_2},b_{12}] = \\ 
\quad & =[a_{21},b_{21}]\cup [a_{22},b_{22}]\cup [a_{23},b_{23}]\cup
[b_{24},b_{24}],%
\end{array}%
\end{equation*}
with length $\delta^{t_2}$, and with distances between knot points: 
\begin{equation*}
\begin{array}{l}
|b_{2j}-a_{2j}|=\delta^{t_2}, j=1,2,3,4, \\ 
|b_{21}-a_{22}|=\delta -2\delta^{t_2}, |b_{22}-a_{23}|=1-2\delta,
|b_{23}-a_{24}|=\delta -2\delta^{t_2}.%
\end{array}%
\end{equation*}
In $m$-th step we get union of $2^m$ segments 
\begin{equation*}
K_m=[a_{m1},b_{m1}]\cup [a_{m2},b_{m2}]\cup ...\cup [a_{m2^m},b_{m2^m}],
\end{equation*}
with length $\delta^{t_m}$. Note, $K_0 \supset K_1 \supset ... \supset K_m
..., \,\,\,l(K_m)=2^m\delta^{t_m}$. Moreover, the Hausdorff measure of $K_m$
with respect to kernel $h(s)=\ln^{-1}\frac{1}{s}$ is equal to 
\begin{equation}
H^h(K_m)=2^mh(\delta^{t_m}/2)=2^m\ln^{-1}\frac{1}{\delta^{t_m}/2}=\frac{2^m}{%
t_m}\ln^{-1}\frac{2^{1/t_m}}{\delta}.  \label{eq:5}
\end{equation}
Put $\displaystyle K=\bigcap^{\infty}_{m=1}K_m.$ If $\displaystyle \frac{2^m%
}{t_m}\leq C <\infty,$ $m=1,2,...,$ then $H^h(K)<\infty$ and by the
well-known property of the logarithm capacity $C(K)=0$. Therefore, in our
case $t_m=4^{m-1}$, the compact set $K$ is polar and there exists a
probability measure $\mu$, $\text{supp} \mu =K$, such that its potential 
\begin{equation*}
U^\mu(z)=\int\ln |z-w|d\mu (w)
\end{equation*}
is harmonic off $K$, subharmonic on $\mathbb{C}^n $, and $U^{\mu}(z)=-\infty$
$\forall z\in K$.

Now we will specifically construct such measure $\mu$. For $%
K_m=[a_{m1},b_{m1}]\cup [a_{m2},b_{m2}]\cup ...\cup [a_{22^m},b_{22^m}]$ we
put 
\begin{equation}
\mu_m=\frac{\delta (a_{m1})+...+\delta (a_{m2^m})+\delta (b_{m1})+...\delta
(b_{m2^m})}{2\cdot 2^m},  \label{eq:6}
\end{equation}
where $\delta (c)$-discrete probably measure, supported in $c$. The sequence 
$\mu_m$ weakly tends to a measure $\mu_m \mapsto \mu$, $\text{supp} \mu =K.$
Let 
\begin{equation*}
U_m^{\mu}(z)=\int \ln |z-w|d\mu_m(w), U^{\mu}(z)=\int \ln |z-w|d\mu(w)
\end{equation*}
be the potentials. We give some estimations to these potentials.

Take $z^0\in\mathbb{C}^n \setminus K,\lambda =dist(z^0,K)>0$. Then by a
well-known integral formula (see \cite{Fed}). 
\begin{equation}
U^{\mu}_m(z^0)=\int\ln |z^0-w|d\mu_m(w)=\int_{0}^{\infty}[\ln
t]d\mu_m(z^0,t)=\int_{\lambda_m}^{\Lambda}[\ln t]d\mu_m(z^0,t),  \label{eq:7}
\end{equation}
where $\mu_m (z^0,t)=\mu_m (B(z^0,t))$, $\,\,B(z^0,t):|z-z^0|\leq t$ is
disk, $\,\,\Lambda =\max \{dist(z^0,0),dist(z^0,1)\}$, $\,\,\lambda_m=\min
\{|z^0-a_{mj}|,|z^0-b_{mj}|:j=1,2,...,2^m\}$ is the distance from $\,z^0\,$
to the knot set $\,\,K_m^{knot}=%
\{a_{m1},b_{m1},a_{m2},b_{m2},...,a_{m2^m},b_{m2^m}\}$, $\,\lambda_m\geq
\lambda$. Integrating by part (\ref{eq:7}) we get 
\begin{equation*}
\begin{array}{ll}
\displaystyle U^{\mu}_m(z^0) & =\int_{\lambda_m}^{\Lambda}[\ln
t]d\mu_m(z^0,t)=\mu_m(z^0,t)\ln t
|_{\lambda_m}^{\Lambda}-\int_{\lambda_m}^{\Lambda}\frac{\mu_m(z^0,t)}{t}dt=
\\ 
\quad & =\ln \Lambda -\int_{\lambda_m}^{\Lambda}\frac{\mu_m(z^0,t)}{t}dt.%
\end{array}
\end{equation*}%
%

Next we will estimate the potentials $U_m^{\mu}(z^0),U^{\mu}(z^0)$ for
nearby to $K$ point $z^0$, say $\lambda_m<1$. Let $c$ is a knot point, such
that $\lambda_m=|z^0-c|$. The cases $c=0$ or $c=1$ are simple and both are
similar one to one. Other cases reduces to these cases by parting knot set $%
\{a_{m1},b_{m1},a_{m2},b_{m2},...,a_{m2^m},b_{m2^m}\}$ two sets: right and
left from $\text{Re}\, z^0$. Therefore, without loss of generality, we
assume that $c=0$ and $\text{Re} \,z^0\leq 0$. In this case, $\mu_m(0,
t-\lambda_m)\leq \mu_m(z^0,t)\leq \mu_m(0, \sqrt{t^2-\lambda_m^2}).$ If we
denote $\mu_m(t)=\mu_m(0,t)$, then 
\begin{equation}
-\int_{\lambda_m}^{\Lambda}\frac{\mu_m(t-\lambda_m)}{t}dt \leq
-\int_{\lambda_m}^{\Lambda}\frac{\mu_m(z^0,t)}{t}dt \leq
-\int_{\lambda_m}^{\Lambda}\frac{\mu_m(\sqrt{t^2-\lambda_m^2})}{t}dt
\label{eq:8}
\end{equation}
It is clear, that 
\begin{equation*}
\mu_m(\delta)=\frac{1}{2},\mu_m(\delta^{t_2})=\frac{1}{2^2}%
,...,\mu_m(\delta^{t_{m-1}})=\frac{1}{2^{m-1}},\mu_m(\delta^{t_m})=\frac{1}{%
2^m}.
\end{equation*}

Therefore,

\begin{equation*}
\mu_m(t)=\frac{1}{2},\text{ if }\delta\leq t <1-\delta;
\end{equation*}
\begin{equation*}
\mu_m(t)=\frac{1}{2^2},\text{ if }\delta^{t_{2}}\leq t
<\delta-\delta^{t_{2}};
\end{equation*}
\begin{equation}
\vdots  \label{eq:9}
\end{equation}
\begin{equation*}
\mu_m(t)=\frac{1}{2^{m-1}},\text{ if }\delta^{t_{m-1}}\leq t
<\delta^{t_{m-2}}-\delta^{t_{m-1}};
\end{equation*}
\begin{equation*}
\mu_m(t)=\frac{1}{2^{m}},\text{ if }\delta^{t_{m}}\leq t
<\delta^{t_{m-1}}-\delta^{t_{m}}.
\end{equation*}

\vskip  0.9 cm Using (\ref{eq:8}) and (\ref{eq:9}) we can give upper and
lower bounds of $U^{\mu}(z)$.

\textbf{a) Upper bound} of the potential $\,U^{\mu}(z)$. We have 
\begin{equation*}
I_m=-\int_{\lambda_m}^{\Lambda}\frac{\mu_m(z^0,t)}{t}dt \leq
-\int_{\lambda_m}^{\Lambda}\frac{\mu_m(\sqrt{t^2-\lambda_m^2})}{t}dt =
-\int_{0}^{\sqrt{\Lambda^2-\lambda_m^2}}\frac{t}{t^2+\lambda_m^2}\mu_m(t)dt -
\end{equation*}
\begin{equation*}
-\int\limits_{0}^{\delta^{t_{m-1}}-\delta^{t_{m}}}\frac{t}{t^2+\lambda_m^2}%
\mu_m(t)dt
-\int\limits_{\delta^{t_{m-1}}-\delta^{t_{m}}}^{\delta^{t_{m-2}}-%
\delta^{t_{m-1}}}\frac{t}{t^2+\lambda_m^2}\mu_m(t)dt -...
-\int\limits_{\delta -\delta^{t_{2}}}^{1-\delta}\frac{t}{t^2+\lambda_m^2}%
\mu_m(t)dt-
\end{equation*}
\begin{equation*}
-\int\limits_{1-\delta}^{1}\frac{t}{t^2+\lambda_m^2}\mu_m(t)dt
-\int\limits_{\delta^{t_{m}}}^{\delta^{t_{m-1}}-\delta^{t_{m}}}\frac{t}{%
t^2+\lambda_m^2}\mu_m(t)dt
-\int\limits_{\delta^{t_{m-1}}}^{\delta^{t_{m-2}}-\delta^{t_{m-1}}}\frac{t}{%
t^2+\lambda_m^2}\mu_m(t)dt -...
\end{equation*}
\begin{equation*}
...-\int\limits_{\delta}^{1-\delta}\frac{t}{t^2+\lambda_m^2}\mu_m(t)dt= -%
\frac{2}{2^{m+1}} \int\limits_{\delta^{t_{m}}}^{\delta^{t_{m-1}}-%
\delta^{t_{m}}}\frac{tdt}{t^2+\lambda_m^2} -\frac{2^2}{2^{m+1}}
\int\limits_{\delta^{t_{m-1}}}^{\delta^{t_{m-2}}-\delta^{t_{m-1}}}\frac{tdt}{%
t^2+\lambda_m^2} -... 
\end{equation*}
\begin{equation}
...-\frac{2^m}{2^{m+1}} \int\limits_{\delta}^{1-\delta}\frac{tdt}{%
t^2+\lambda_m^2} .
\end{equation}

Therefore 
\begin{equation*}
I_m \leq \frac{2}{2^{m+2}}\ln \frac{\lambda_m^2+\delta^{2t_{m}}}{%
\lambda_m^2+(\delta^{t_{m-1}}-\delta^{t_{m}})^2} +\frac{2^2}{2^{m+2}}\ln 
\frac{\lambda_m^2+\delta^{2t_{m-1}}}{\lambda_m^2+(\delta^{t_{m-2}}-%
\delta^{t_{m-1}})^2} +...
\end{equation*}
\begin{equation*}
...+\frac{2^m}{2^{m+1}}\ln \frac{\lambda_m^2+\delta^2}{\lambda_m^2+(1-%
\delta)^2} =\frac{1}{2^{m}}\ln\frac{\lambda_m^2+\delta^{2t_{m}}}{%
\lambda_m^2+\delta^{2t_{m-1}}} +\frac{2}{2^{m}}\ln \frac{\lambda_m^2+%
\delta^{2t_{m-1}}}{\lambda_m^2+\delta^{2t_{m-2}}} +...
\end{equation*}
\begin{equation}
...+\frac{2^{m-1}}{2^{m}}\ln \frac{\lambda_m^2+\delta^2}{\lambda_m^2+1}%
+o(\delta^{t_{m-1}}) = \frac{1}{2^{m}}\ln (\lambda_m^2+\delta^{2t_{m}})\frac{%
1}{2^{m}}\ln (\lambda_m^2+\delta^{2t_{m-1}})+  \label{eq:10}
\end{equation}
\begin{equation*}
+\frac{2}{2^{m}}\ln (\lambda_m^2+\delta^{2t_{m-2}})+...\frac{2^{m-2}}{2^{m}}%
\ln (\lambda_m^2+\delta^2) -\frac{2^{m-1}}{2^{m}}\ln
(\lambda_m^2+1)+o(\delta^{t_{m-1}}).
\end{equation*}

Let $k=k(z^0)$ be the smallest natural number, such that $\delta^{t_{k}}\leq
\lambda_m.$ We part the last sum in (\ref{eq:10}) into two sums: by $k\leq
j\leq m$ ( $\delta^{t_{j}}\leq\lambda_m$) and by $j<k$ \newline
( $\delta^{t_{j}}>\lambda_m$). For the first sum, by $\delta^{t_{j}}\leq%
\lambda_m$, we write

\begin{equation*}
\frac{1}{2^m}\ln (\lambda_m^2+\delta^{2t_{m}}) +\frac{1}{2^m}\ln
(\lambda_m^2+\delta^{2t_{m-1}}) +\frac{2}{2^m}\ln
(\lambda_m^2+\delta^{2t_{m-2}})+... 
\end{equation*}
\begin{equation*}
...+\frac{2^{m-k-1}}{2^m}\ln (\lambda_m^2+\delta^{2t_{m}})\leq \frac{%
1+2+...+2^{m-k-1}}{2^m}\ln (2\lambda_m^2)= \frac{2^{m-k}-1}{2^m}\ln
(2\lambda_m^2)\leq \frac{1}{2^m}\ln (2\lambda_m^2).
\end{equation*}

Since $t_k =4^{k-1}$ and $\delta^{t_k}\leq \lambda_m$, then $2^k\geq \sqrt{%
\frac{\ln \lambda_m}{\ln\delta}}$. Therefore, the first sum is not greater
than $\frac{1}{2^k}\ln 2\lambda_m^2\leq \sqrt{\ln\frac{1}{\delta}}\frac{%
2\ln\lambda_m+\ln 2}{\sqrt{\ln\frac{1}{\lambda_m} }} $.

For the second sum, by $\delta^{t_j}>\lambda_m$, we have, 
\begin{equation*}
\frac{2^{m-k}}{2^m}\ln (\lambda_m^2+\delta^{2t_{k-1}})+...+\frac{2^{m-2}}{2^m%
}\ln (\lambda_m^2+\delta^2)-\frac{2^{m-1}}{2^m}\ln (\lambda_m^2+1)\leq -%
\frac{1}{2}\ln (\lambda_m^2+1)+
\end{equation*}
\begin{equation*}
+\frac {1}{2^2}\ln (\lambda_m^2+\delta^2)+...+\frac{1}{2^k}\ln
(\lambda_m^2+\delta^{2t_{k-1}}) \leq -\frac{1}{2}\ln (\lambda_m^2+1)+\frac{1%
}{2^2}\ln (2\delta^2)+... +\frac{1}{2^k}\ln (2\delta^{2t_{k-1}}) \leq 
\end{equation*}
\begin{equation*}
\leq -\frac{1}{2}\ln (\lambda_m^2+1)+\frac{1}{2^2}\ln (2\delta^2)+...+\frac{1%
}{2^{k+1}}\ln (2\delta^2)=-\frac{1}{2}\ln (\lambda_m^2+1)+\frac{1}{2}\ln
(2\delta^2).
\end{equation*}

Therefore, for large enough $m$ is true the following estimation 
\begin{equation}
U^{\mu}_m(z^0)\leq \sqrt{\ln \frac{1}{\delta}}\,\,\frac{\ln \lambda_m+\ln 2}{%
\sqrt{\ln \frac{1}{\lambda_m}}}-\frac{1}{2}\ln (\lambda_m^2+1)+\ln\Lambda+%
\frac{1}{2}\ln 2\delta+o(\delta^{t_{m-1}}).  \label{eq:11}
\end{equation}
For arbitrary $z^0\in\mathbb{C}^n \setminus K$ the estimation (\ref{eq:11})
will be

\begin{equation*}
U^{\mu}_m(z^0)\leq 2\sqrt{\ln \frac{1}{\delta}}\,\,\frac{\ln
dist(z^0,K_m^{knot})+\ln 2}{\sqrt{\ln \frac{1}{dist(z^0,K_m^{knot})}}}\,\,-
\end{equation*}
\begin{equation}
-\frac{1}{2}\ln (dist^2(z^0,K_m^{knot})+1)+\ln\Lambda+\frac{1}{2}\ln
2\delta+o(\delta^{t_{m-1}}).  \label{eq:12}
\end{equation}
Tending $m \to \infty$ in (4.10) we take

\begin{equation*}
U^{\mu}(z^0)\leq 2\sqrt{\ln \frac{1}{\delta}}\,\,\frac{\ln dist(z^0,K)+\ln 2%
}{\sqrt{\ln \frac{1}{dist(z^0,K)}}}\,\,-
\end{equation*}
\begin{equation}
-\frac{1}{2}\ln (dist^2(z^0,K)+1)+\ln\Lambda+\frac{1}{2}\ln 2\delta.
\label{eq:13}
\end{equation}
From (\ref{eq:13}), in particular, follows, that $U^{\mu}(z^0)=-\infty,
\forall z^0\in K$.

\textbf{b) Lower bound.} As above, we have:

\begin{equation*}
I_m=-\int\limits_{\lambda_m}^{\Lambda}\frac{\mu_m(z^0,t)}{t}dt \geq
-\int\limits_{\lambda_m}^{\Lambda}\frac{\mu_m(t-\lambda_m)}{t}dt =
-\int\limits_{0}^{\Lambda-\lambda_m}\frac{\mu_m(t)}{t+\lambda_m}dt=
\end{equation*}

\begin{equation*}
= -\int\limits_{0}^{\delta^{t_{m-1}}-\delta^{t_{m}}}\frac{\mu_m(t)}{%
t+\lambda_m}dt
-\int\limits_{\delta^{t_{m-1}}-\delta^{t_{m}}}^{\delta^{t_{m-2}}-%
\delta^{t_{m-1}}}\frac{\mu_m(t)}{t+\lambda_m}dt -... -\int\limits_{\delta
-\delta^{t_{2}}}^{1-\delta}\frac{\mu_m(t)}{t+\lambda_m}dt
-\int\limits_{1-\delta}^{1}\frac{\mu_m(t)}{t+\lambda_m}dt \geq
\end{equation*}
\begin{equation*}
\geq -\frac{2}{2^{m+1}} \int\limits_{0}^{\delta^{t_{m-1}}-\delta^{t_{m}}}%
\frac{dt}{t+\lambda_m} -\frac{2^2}{2^{m+1}} \int\limits_{\delta^{t_{m-1}}-%
\delta^{t_{m}}}^{\delta^{t_{m-2}}-\delta^{t_{m-1}}}\frac{dt}{t+\lambda_m}
-... -\frac{2^m}{2^{m+1}} \int\limits_{\delta-\delta^{t_2}}^{1-\delta}\frac{%
dt}{t+\lambda_m} -\frac{2^{m+1}}{2^{m+1}} \int\limits_{1-\delta}^{1}\frac{dt%
}{t+\lambda_m}=
\end{equation*}
\begin{equation*}
= -\frac{1}{2^{m}}\ln \frac{\lambda_m+\delta^{t_{m-1}}-\delta^{t_{m}}}{%
\lambda_m} -\frac{1}{2^{m-1}}\ln \frac{\lambda_m+\delta^{t_{m-2}}-%
\delta^{t_{m-1}}}{\lambda_m+\delta^{t_{m-1}}-\delta^{t_{m}}}-... -\frac{1}{2}%
\ln \frac{\lambda_m+1-\delta}{\lambda_m+\delta-\delta^{t_{2}}} -\ln \frac{%
\lambda_m+1}{\lambda_m+1-\delta}=
\end{equation*}
\begin{equation*}
=\frac{\ln\lambda_m}{2^m} +\frac{\ln
(\lambda_m+\delta^{t_{m-1}}-\delta^{t_{m}})}{2^m} +\frac{\ln(\lambda_m+%
\delta^{t_{m-2}}-\delta^{t_{m-1}})}{2^{m-1}}+... +\frac{\ln(\lambda_m+1-%
\delta)}{2}-\ln(\lambda_m+1).
\end{equation*}

Therefore 
\begin{equation*}
I_m \geq -\ln(\lambda_m+1) +\frac{\ln(\lambda_m+1-\delta)}{2} +\frac{%
\ln(\lambda_m+\delta-\delta^{t_{2}})}{2^2}+
\end{equation*}
\begin{equation*}
+\frac{\ln(\lambda_m+\delta^{t_2}-\delta^{t_{3}})}{2^3} +... +\frac{%
\ln(\lambda_m+\delta^{t_{m-2}}-\delta^{t_{m-1}})}{2^{m-1}} +\frac{\ln
(\lambda_m+\delta^{t_{m-1}}-\delta^{t_{m}})}{2^m} \frac{\ln\lambda_m}{2^m}%
\geq
\end{equation*}
\begin{equation*}
\geq -\ln(\lambda_m+1) +\frac{\ln(1-\delta)}{2} +\frac{\ln(\delta-%
\delta^{t_{2}})}{2^2} +... +\frac{\ln(\delta^{t_{k-1}}-\delta^{t_{k}})}{2^k}+
\end{equation*}
\begin{equation*}
+\frac{\ln \lambda_m}{2^{k+1}} +... +\frac{\ln \lambda_m}{2^{m-1}} +\frac{%
\ln \lambda_m}{2^{m}} +\frac{\ln \lambda_m}{2^{m}} = c(k) +\frac{\ln
\lambda_m}{2^{k}}\left(1-\frac{1}{2^{m-k}}\right),
\end{equation*}
where $c(k)=const$, independent of $m$. Hence, for any fixed $k\in\mathbb{C}%
^n$ we have 
\begin{equation}
U_m^{\mu}(z^0)\geq\ln \Lambda +c(k)+\frac{\ln\lambda_m}{2^k}\left(1-\frac{1}{%
2^{m-k}}\right).  \label{eq:14}
\end{equation}
As above we can prove (\ref{eq:14}) for arbitrary $z^0\notin K$: 
\begin{equation}
U_m^{\mu}(z^0)\geq 2\left(\ln \Lambda +c(k)+\frac{\ln dist(z^0,K_m^{knot})}{%
2^k}\left(1-\frac{1}{2^{m-k}}\right)\right).  \label{eq:15}
\end{equation}
Tending $m\rightarrow\infty$ from (\ref{eq:15}) we conclude, that for any $%
\varepsilon >0$ there exists constant $c(\varepsilon )>-\infty$: 
\begin{equation*}
U_m^{\mu}(z^0)\geq c(\varepsilon ) + \varepsilon\ln dist(z^0,K),\forall
z^0\in\mathbb{C}^n.
\end{equation*}
Theorem is proved.
\end{proof}

Now we can proceed with our example,

\begin{example}
We consider the manifold $X=\overline{\mathbb{C}}\setminus K$, where $K$ is
compact, built in the previous point. As special exhaustive function we put $%
\phi (z)=-U^{\mu}(z)$. Then $\phi (z)$ is harmonic on $X\setminus \{\infty\}$%
, $\phi(\infty)=-\infty$ and $\phi (z)\rightarrow\infty$ as $z\rightarrow K$%
. Therefore, $(X,\phi)$ is $S^*$-parabolic.

Polynomials on $X$ are functions $f\in O(X)$ for which $\ln |f|\leq C+d\phi
(z),d\in\mathbb{N}$. We show that this like functions are trivial, i.e. $%
f=const$. It follows, that on $X$ there are not nontrivial polynomials, $X$
is nonregular.

This easily follows from the next Proposition, which seems clear and there
is a proof of them: let $K $ is a polar compact on the complex plane $%
\mathbb{C}, $ where $U\supset K$ is some neighborhood. If $f(z)\in
O(U\setminus K)$ and 
\begin{equation}
\overline{\lim_{z\rightarrow K}}\,\,|f(z)|\cdot dist(z,K)=0, \,
\label{eq:16}
\end{equation}
then $f(z)\in O(U)$.

Since we cannot find the proof of this proposition, we provide it for our
compact $K$. Let $f\in O(X):\ln |f|\leq C+k\phi (z)$.

First we take a closed curve $\gamma =\gamma_m$, containing within itself
the $K\subset K_m=[a_{m1},b_{m1}]\cup [a_{m2},b_{m2}]\cup ... \cup
[a_{22^m},b_{22^m}]:\gamma$ bounds above by a part of $\{\text{Im} z=r\},r>0$%
, below by $\{\text{Im} z=-r\}$ and from the sides by a part $\{\text{Re}
z=a_{mj}-r\},\{\text{Re} z=b_{mj}+r\}$. The length of $\gamma $ is equal 
\begin{equation}
l(\gamma )=2\cdot 2^m(\delta^{t_m}+2r)+2\cdot 2^mr=3\cdot
2^{m+1}r+2^{m+1}\delta^{t_m}.  \label{eq:17}
\end{equation}
To complete of the proof we write the Cauchy formula 
\begin{equation}
f(z)=\frac{1}{2\pi i}\int\limits_{|\xi |=2}\frac{f(\xi )}{\xi -z}d\xi -\frac{%
1}{2\pi i}\int\limits_{\gamma }\frac{f(\xi )}{\xi -z}d\xi,\,\, z\in
B(0,2)\setminus \hat{\gamma},  \label{eq:18}
\end{equation}
where $\hat{\gamma}$ is the polynomial convex hull of $\gamma$.

For second integral of (\ref{eq:18}) we have 
\begin{equation*}
\left|\int\limits_{\gamma }\frac{f(\xi )}{\xi -z}d\xi\right| \leq \frac{%
||f||_{\gamma}}{dist (z,\gamma)}\,l(\gamma ) \leq \frac{e^{C+k||\phi||_{%
\gamma}}}{dist (z,\gamma)}(3\cdot 2^{m+1}r+2^{m+1}\delta^{t_m}) \leq C_1
e^{k ||\phi||_{\gamma}}(2^{m+3}r+2^{m+1}\delta^{t_m}).
\end{equation*}

According to (\ref{eq:4}) for arbitrary fixed $\varepsilon >0 $ there exists 
$\gamma = \gamma _m$ such,that $||\phi||_{\gamma} <-\varepsilon \ln dist
(\gamma, K)$. Therefore, $\left|\int\limits_{\gamma }\frac{f(\xi )}{\xi -z}%
d\xi\right| \leq C_2r^{-\varepsilon k}2^m(r+\delta^{t_m}). $

Now we choose $\varepsilon =1/2k$ and $r = 1/2^{4m}$. Then $r^{-\varepsilon
k}2^m(r+\delta^{t_m})=\frac{1}{2^m}+2^{3m}\delta^{t_m}\rightarrow 0$ as $%
m\rightarrow\infty$. We see that, the second integral in (\ref{eq:18}) tends
zero, which means the function 
\begin{equation*}
f(z)=\frac{1}{2\pi i}\int\limits_{|\xi |=R}\frac{f(\xi )}{\xi -z}d\xi
\end{equation*}
and holomorphic in $|z|<R$. Consequently $f\in O(\overline{\mathbb{C}}),$
i.e. $f\equiv const$.
\end{example}


\section{Regular parabolic manifolds.}

\label{sec4} 

As we have seen in section \ref{sec3} not every parabolic manifold has a
large supply of polynomials. On the other hand most important examples of
parabolic manifolds like affine algebraic submanifolds (with their canonical
special exhaustion function), complements of zero sets of Weierstrass
polynomials (see \cite{AS2}) do have a rich class of polynomials, namely in
these examples polynomials are dense in the corresponding spaces of analytic
functions. 

\begin{example}
\textbf{Algebraic set} $X\subset \mathbb{C}^N$, $\dim A=n$. In this case by
the well-known theorem of W. Rudin \cite{R}, we can assume, that (after an
appropriate transformation) 
\begin{equation*}
X\subset \{ w=(w^{\prime },w^{\prime \prime
})=(w_1,...,w_n,w_{n+1},...,w_N):||w^{\prime \prime }||<A(1+||w^{\prime}||^
B \},
\end{equation*}
where $A,B$ are constants. Then the restriction $\rho |_X$ of the function $%
\rho (w)=\ln ||w^{\prime }||$ may be special exhaustion function on $X$. It
is clear, that polynomials on $X$ are restrictions to $X$ of polynomials $p
(w^{\prime },w^{\prime \prime })$. Therefore, $\mathcal{P}_{\rho}(X)$ is
dense in $O(X)$.
\end{example}

\begin{example}
\textbf{Complement of Weierstrass algebroid set} (see Theorem \ref{thm27}).
Let 
\begin{equation*}
A=\{z=(^{\prime }z,z_n)=(z_1,z_2,...,z_n)\in\mathbb{C}^n:F(^{\prime
}z,z_n)=z_n^k+f_1(^{\prime }z)z_n^{k-1}+...+f_k(^{\prime }z)=0\}
\end{equation*}
be a Weierstrass polynomial set, where $f_j\in O(\mathbb{C}^{n-1})$ are
entire functions, $j = 1,2,...,k$, $k > 1$. Then $X=\mathbb{C}^n\setminus A$
with exhaustion function $\rho (z)=-\ln |F(z)|+\ln (|^{\prime}z|+|F(z)-1|^2)$
is $S^*$-parabolic. If $p(z,\tau)$ is a polynomial in $\mathbb{C}^{n+1}$,
then $p(z,1/F(z))$ is a polynomial on $X=\mathbb{C}^n \setminus A$.

It is not difficult to prove, that $\left\{p(z,1/F(z))\right\}_p $ is dense
in $O(X)$.
\end{example}

Motivated by these examples, we give the following definiton:

\begin{definition}
${S}^*$-parabolic manifold $(X,\rho)$ calls \textit{regular} in case if the
space of all polynomials $\mathcal{P}_{\rho}(X)$ is dense in $O(X)$.
\end{definition}

Our next example shows that non triviality of the polynomial space $\mathcal{%
P}_{\rho}(X)$ does not always guarantee the regularity of $X$.

\begin{example}
We add to compact $K$, from example 4.2 one more point:\newline
$E=K\cup \{z^0\}$, $z^0\notin K$. The manifold $X=\overline{\mathbb{C}}%
\setminus E$ with exhaustive function $\rho (z)=-U^{\mu}(z)-\ln |z-z^0|$ be $%
S^*$-parabolic. On $X$ there are polynomials, an example, $f(z)=(z-z^0)^m$,
but the space of all polynomials $\mathcal{P}_{\rho}$ is not dense in $O(X)$%
: the function $f(z)=\frac{1}{z-z^{\prime }}$, where $z^{\prime }\in K,$
cannot be approximated by polynomials.
\end{example}

In search for more examples of ${S}$-parabolic manifolds one may look at
closed complex submanifolds of regular ${S}^{*}$-parabolic manifolds. Since
such manifolds are in particular parabolic, there exits, in view of Theorem %
\ref{thm210}, continuous linear extension operators for analytic functions
on this submanifold to the ambient space. However the mere existence of
continuous extension operators will not, in general give regularity as the
example, in the previous section shows.

Recall that for a ${S}^{*}$-parabolic manifold $\left( X, \rho \right)$ we
will always consider, unless otherwise stated, the canonical grading on $%
O\left( X\right) $ given by $\rho ,$ and for a closed complex submanifold $V$
of $X$ we will provide $O\left( V \right) $ with the induced grading, i.e.
the grading coming from the sup norms on $V \cap \left( z:\rho \left(
z\right) \leq k\right) $, $k=1,2,...\,.$ With this convention we have:

\begin{prop}
Let $\left( X,\rho \right) $ be a regular ${S}^{*}$-parabolic Stein manifold
and let $V$ be a closed complex submanifold of $X.$ If there exits a tame
linear extension operator from $O\left( V\right) $ into $O\left( X\right) $
then $V$ becomes a regular ${S}^*$-parabolic manifold.
\end{prop}

\begin{proof}
Fix a continuous linear extension operator $\,E:$ $O\left( V\right)
\rightarrow O\left( X\right) $ with the property: 
\begin{equation*}
\exists A>0 \,\,\text{such, that } \forall k\text{ }\exists \text{ }%
C_{k}>0:\left\Vert E\left( f\right) \right\Vert _{k}\leq C_k \left\Vert
f\right\Vert _{k+A}\text{ }\forall f\in O\left( V\right) .
\end{equation*}%
Let as usual 
\begin{eqnarray*}
\mathcal{A} = \left \{ u(z) \in psh \left( V\right) : \,\,u(z) \leq L_u +
\rho^+ (z) \,\,\forall z \in V, \,\,\,u \leq 0 \,\,\text { on } \,\, V\cap
D_{A+2} \right\},
\end{eqnarray*}%
where $D_{k}=\left( z\in X:\,\,\rho \left( z\right) <k\right) $ , $%
k=1,2,...\,.$

Fix a $\,\,u\in \mathcal{A}.$ In view of Lelong Bremermann Lemma \cite{B}, $u
$ has a represantation of the form:%
\begin{equation*}
u\left( z\right) = \overline{\lim_ {\xi \rightarrow z}}\,\, {\overline{%
\lim_{m \rightarrow \infty}}} \,\,\frac{\ln \left\vert f_{m}\left( \xi
\right) \right\vert }{\alpha _{m}}
\end{equation*}%
for some $f_{m}\in O\left( V\right) \,\,$ and $\,\,\alpha _{m}\in \mathbb{N},
$ $\,m=1,2,...\,.$

In view of Hartog's lemma, for each $k=1,2,...$, we can find a constant 
\newline
$C=C\left( k\right), $ such that 
\begin{equation*}
\left\Vert f_{m}\right\Vert _{k}\leq Ce^{\left( k+L+1\right) \alpha
_{m}},m=1,2,..., \,\,\, L=L_u.
\end{equation*}%
Hence%
\begin{equation*}
\left\Vert E\left( f_{m}\right) \right\Vert _{k}\leq Ce^{\left(
k+A+L+1\right) \alpha _{m}},\,\,m=1,2,\text{...},
\end{equation*}%
and so the sequence of plurisubharmonic functions 
\begin{equation*}
\left\{ \frac{\ln \left\vert E\left( f_{m}\left( \xi \right) \right)
\right\vert }{\alpha _{m}}\right\} _{m}
\end{equation*}
is a locally bounded from above family. Let 
\begin{equation*}
\widetilde{u}\left( z\right) =\overline{\lim _{\xi \rightarrow z}}\,\,%
\overline{\lim _{m \rightarrow \infty}}\,\,\frac{\ln \left\vert E\left(
f_{m}\right) \left( \xi \right) \right\vert }{\alpha _{m}}.
\end{equation*}%
The function $\widetilde{u}$ defines a plurisubharmonic function on $X$ and
has the growth estimate:%
\begin{equation*}
\widetilde{u}(z)\leq \rho (z) +A+L+2,
\end{equation*}%
in view of the maximality of $\rho.$ Since, the Green function $V^* (z, 
\overline {D}_1)$ on $X$ is equal $[\rho -1]^+,$ then 
\begin{equation*}
\widetilde{u}(z)\leq V^* (z, \bar {D_1})+C_0, \,\,z\in X.
\end{equation*}

By construction on $V$ we have $u\leq \widetilde{u}|_V.$ It follows that%
\begin{equation*}
u\left( z\right) \leq V^{\ast }\left( z,\overline{D}_{1}\right) +C_{0}\text{ 
}\forall \, z\in V \,\,\text{and } u \in \mathcal{A}.
\end{equation*}
In particular the family $\mathcal{A}$ is a locally bounded from above of
plurisubharmonic functions on $V.$ In view of the above considerations the
free envelope 
\begin{equation*}
\tau \left( z\right) = \overline {\lim _{\xi \rightarrow z}} \sup_{\ u \in 
\mathcal{A}} u \left( \xi \right)
\end{equation*}%
defines a plurisubharmonic function on $V$ that is maximal outside a compact
set $\,\,V \bigcap \overline {D}_ {A+2}\,\,$ and satisfies the estimates:%
\begin{equation*}
\exists \text{ }C>0:\,\,\rho \left( z\right) \leq \tau \left( z\right) \leq
\rho \left( z\right) +C,
\end{equation*}%
since $[\rho -(A+2)]_V \in \mathcal{A}.$ Hence\ $\tau $ provides a special
exhaustion function for $V.$ Moreover since the restriction of a $\rho -$%
polynomial to $V$ is a $\tau -$polynomial, the regularity of $V$ follows.
\end{proof}

\begin{remark}
The existence of a tame linear extension operator as above is of course
related to the tame splitting of tame short exact sequence:%
\begin{equation*}
0\rightarrow I\rightarrow O\left( X\right) \overset{R}{\rightarrow }O\left(
V\right) \rightarrow 0,
\end{equation*}%
where $R$ is the restriction operator and $I$ is the ideal sheaf of $V$ with
the subspace grading induced from $O\left( X\right) .$ Tame splitting of
short exact sequences in the category of graded Fr\'{e}chet spaces were
studied by various authors. We refer the reader to \cite{pop} for a survey
and for structural conditions on the underlying Fr\'{e}chet nuclear spaces
which ensure that short exact sequences in this category split.
\end{remark}

\begin{remark}
It was shown in \cite{AKT2} that in $\mathbb{C}^{N}$ closed complex
submanifolds that admit tame extension operators are precisely the affine
algebraic submanifolds of $\mathbb{C}^{N}.$ Since there are non algebraic
regular ${S}^{*}$-parabolic Stein manifolds of $\mathbb{C}^{N},$ the
statement of the Proposition is not an if and only if statement.
\end{remark}

Our next result deals with the linear topological structure of the graded
space of analytic functions $\left( O\left( X\right) ,{\rho }\right) $ on a $%
S^*-$parabolic Stein manifold $\left( X,{\rho }\right) $\ . Recall that for
a given $S^{\ast }-$ parabolic Stein manifold $\left( X,{\rho }\right) $, we
will always assume that the special exhaustion function $\rho $ is maximal
outside a compact set that lies in $\left\{ z:\,\,\rho \left( z\right)
<0\right\} $ and equip the Fre\'{c}het space $O\left( X\right) $ with the
grading $\left( \left\Vert {\cdot}\right\Vert _{k}\right) _{k=1}^{\infty }:$ 
\begin{equation*}
\left\Vert f\right\Vert _{k} = \sup_{z \in D_{k}}\left\vert f\left( z\right)
\right\vert ,
\end{equation*}
where $D_{k} = \left( z: \,\rho \left( z\right) < k\right) $ , $k=$1,2,....
. On $O\left( \mathbb{C}^{n}\right) $ the canonical grading will be the one
coming from the norm system 
\begin{equation*}
\left\Vert f\right\Vert _{k} = \sup_{\left\Vert z\right\Vert \leq e^{k\text{ 
}}}\left\vert f\left( z\right) \right\vert, \,\,\,k=1,2,...\,.
\end{equation*}
We have seen that with a suitable special exhaustion function $\rho ,$ $%
\left( O\left( X\right) ,{\rho }\right) $ is tamely isomorphic to $O\left( 
\mathbb{C}^{n}\right) $ with the canonical grading. Unfortunately tame
isomorphisms between $S^{\ast }-$ parabolic Stein manifolds do not
necessarily map polynomials into polynomials even when the spaces are
regular as the multiplication operator with the exponential function on $%
O\left( \mathbb{C}\right) $ shows. However our next result states that for a
regular $S^{\ast }-$ parabolic Stein manifold $\left( X,{\rho }\right) $
there exits a positive constant $C$ and a tame isomophism $T$ from $O\left( 
\mathbb{C}^{n}\right) $ , $n=\dim X$ , onto $\left( O\left( X\right) ,{C\rho 
}\right) $ that maps polynomials into $\rho-$polynomials.

In the proof below we will repeatedly use a fact from functional analysis,
namely the \textit{Dynin-Mitiagin }theorem which states that if a nuclear Fre%
\'{c}het space $\left( Y,{\left\Vert _{{\cdot}}\right\Vert _{k}}\right) $
has a basis $\left\{ g_{m}\right\}$, then it is isomorphic, via the
correspondence 
\begin{equation*}
\sum x_{m}g_{m}\leftrightarrow \left( x_{m}\right) _{m}, 
\end{equation*}%
to the K\"{o}the space:%
\begin{equation*}
\left( K,{\left\vert {\cdot}\right\vert _{k}}\right) = \left\{ x=\left(
x_{m}\right) _{m}:\left\vert x\right\vert _{k} = \sum \left\vert
x_{m}\right\vert \left\Vert g_{m}\right\Vert _{k}<\infty ,\forall k
=1,2,...\right\} . 
\end{equation*}

\bigskip

As usual, for sequences of real numbers $\left\{ \alpha _{k}\right) $ and $%
\left( \beta _{k}\right\} $ the notation $\alpha _{k}\prec \beta _{k}$ means
that there exits a constant $c>0$ that does not depend upon $k$, such that $%
\alpha _{k}\prec c\beta _{k}$ , $\forall k.$

\begin{theorem}
Let $\left( X,{\rho }\right) $ be a regular $S^{\ast }-$ parabolic Stein
manifold. There exits a polynomial basis $\left\{ p_{m}\right\} $ \ for $%
O\left( X \right) $ and a $C>0$, such that the linear transformation $\,T\,$
defined through $\,T\left( p_{m}\right) =$ $z^{\sigma \left( m\right) },$ $%
\,\,m=1,2,...,\,\,$ gives a tame isomorphism between $\left( O\left(
X\right),\,{C \rho }\right) \,\,$ and $\,\,O\left( C^{n}\right) $ with the
usual grading.
\end{theorem}

\begin{proof}
We choose a Hilbert \ space $H_{0}$ with%
\begin{equation*}
O \left( \left\{ z:\,\,\rho \leq 0 \right\} \right) \hookrightarrow
H_{0}\hookrightarrow O \left( \left\{ z:\,\,\rho <0 \right\} \right) \cap
C\left( \left\{ z:\,\,\rho \leq 0 \right\} \right). 
\end{equation*}%
In view of Corollary 1 of \cite{AS2}, and the construction of the proof of
Th.1.5 \cite{V3} on which the proof of the corollary depends, we can without
loss of generality assume that there is a tame isomorphism $S:\,\,O\left( 
\mathbb{C}^{n}\right) \rightarrow \left( O\left( X\right) ,{\rho }\right)$ 
such that the sequence $\left\{ \ f_{m} = S\left( z^{\sigma \left( m\right)
}\right) \right\} $ forms an orthonormal basis for $H_{0}.$

Now we will choose and fix a bijection $\sigma ,$ between $\mathbb{N}$ and $%
\mathbb{N}^{n}$ satisfying $\left\vert \sigma \left( n\right) \right\vert
\leq $ $\left\vert \sigma \left( n+1\right) \right\vert ,$ $\forall n\in 
\mathbb{N}$. Observe that the identity operator gives a tame isomorphism
between $O\left( \mathbb{C}^{n}\right) $ with the canonical grading and $%
\left( O\left( \mathbb{\mathbb{C}}^{n}\right) ,{\left\vert {\cdot}%
\right\vert _{k}}\right), $ where 
\begin{equation*}
\left\vert f\right\vert _{k} = \sum_{n}\left\vert x_{n}\right\vert \text{ }%
e^{k\left\vert \sigma \left( n\right) \right\vert },\text{ \ }\forall \text{ 
}f=\sum_{s}x_{s}z^{\sigma \left( s\right) }\in O\left( \mathbb{C}^{n}\right) 
\end{equation*}%
in view of the Cauchy estimates.

In this case tameness of $S$ provides a positive integer $A,$ such that for
all $k=1,2,...$ 
\begin{equation*}
\left\Vert S\left( f\right) \right\Vert _{k}\prec \sum_{s}\left\vert
x_{s}\right\vert \text{ }e^{\left( k+A\right) \left\vert \sigma \left(
s\right) \right\vert }, 
\end{equation*}%
\begin{equation*}
\sum_{s}\left\vert x_{s}\right\vert \text{ }e^{k\left\vert \sigma \left(
s\right) \right\vert }\prec \left\Vert S\left( f\right) \right\Vert _{k+A}, 
\end{equation*}%
where as usual $\left\Vert f\right\Vert _{k } = \sup_{z\in D_{k }}\left\vert
f\left( z\right) \right\vert ,$ and $\ D_{k } = \left( z:\rho \left(
z\right) \leq k \right).\,$ Since the sequence $\left\{ f_{m}\right \} $
constitutes a basis for $O\left( X\right) ,$ there is a $C_{1}>0$ and $k_{1},
$ so that $\ ,$ 
\begin{equation*}
\sum_{m}\left\vert \beta _{m}\right\vert \left\Vert f_{m}\right\Vert
_{1}\prec \sum_{m}\left\vert \beta _{m}\right\vert e^{\left( 1+A\right)
\left\vert \sigma \left( m\right) \right\vert }\leq C_{1}\left\Vert
\sum_{m}\beta _{m}f_{m}\right\Vert _{k_{1}} 
\end{equation*}%
for every $\ f=\sum_{m}\beta _{m}f_{m}$ $\in O\left( X\right) .$ We choose,
using regularity, polynomials $p_{m},$ $m=1,2,... \,\,$so that 
\begin{equation*}
\left\Vert f_{m}-p_{m}\right\Vert _{m}\leq e^{\left\vert \sigma \left(
m\right) \right\vert },\text{ \ \ }m=1,2,... 
\end{equation*}%
and%
\begin{equation*}
\left\Vert f_{m}-p_{m}\right\Vert _{k_{1}}\leq \frac{1}{2C_{1}}\left\Vert
f_{m}\right\Vert _{1},\text{ \ \ }m=1,2,...\,\text{.} 
\end{equation*}%
For $k>A+1$ and $m\geq k,$ 
\begin{equation*}
\left\Vert p_{m}\right\Vert _{k}\leq \left\Vert f_{m}\right\Vert
_{k}+\left\Vert f_{m}-p_{m}\right\Vert _{m}\leq \left\Vert f_{m}\right\Vert
_{k}+e^{\left\vert \sigma \left( m\right) \right\vert }\prec \left\Vert
f_{m}\right\Vert _{k}. 
\end{equation*}%
Hence for every $k$ large enough , there is a $c_{k}>0$ such that 
\begin{equation*}
\left\Vert p_{m}\right\Vert _{k}\leq c_{k}\left\Vert f_{m}\right\Vert _{k},%
\text{ }\forall m. 
\end{equation*}%
It follows that the operator $Q$ defined by, 
\begin{equation*}
Q\left( \sum_{m}\beta _{m}f_{m}\right) = \sum_{m}\beta _{m}p_{m} 
\end{equation*}%
defines a continuous linear operator from $O\left( X\right) $ into itself.
Moreover for a given $g=\sum_{m}\theta _{m}f_{m}$ in $O\left( X\right) $ and 
$k$ large enough$,$ 
\begin{equation*}
\left\Vert \left( Q-I\right) \left( g\right) \right\Vert _{k}=\left\Vert
\left( Q-I\right) \left( \sum_{m}\theta _{m}f_{m}\right) \right\Vert
_{k}\leq \sum_{m}\left\vert \theta _{m}\right\vert \left\Vert
f_{m}-p_{m}\right\Vert _{k} \leq 
\end{equation*}%
\begin{equation*}
\leq \sum\limits_{m=1}^{k}\left\vert \theta _{m}\right\vert \left\Vert
f_{m}-p_{m}\right\Vert _{k}+\sum\limits_{n=k+1}^{\infty }\left\vert \theta
_{m}\right\vert \left\Vert f_{m}-p_{m}\right\Vert _{k} \leq 
\end{equation*}%
\begin{equation*}
\leq \sup_{1\leq m\leq k}\left( \frac{\left\Vert f_{m}-p_{m}\right\Vert _{k}%
}{\left\Vert f_{m}\right\Vert _{1}}\right) \sum\limits_{m=1}^{k}\left\vert
\theta _{m}\right\vert \left\Vert f_{m}\right\Vert
_{_{1}}+\sum\limits_{m=k+1}^{\infty }\left\vert \theta _{m}\right\vert
e^{\left\vert \sigma \left( m\right) \right\vert }\prec \left\Vert
g\right\Vert _{k_{1}}. 
\end{equation*}%
In view of nuclearity of $O\left( X\right) ,$ the above estimates imply that 
$Q-I$ is a compact operator. In particular $Q$ is Fredholm.

Now suppose there is an $f=\sum\limits_{m}d_{m}f_{m},\,\,$such that $Q\left(
f\right) =0.\,\,$We estimate:%
\begin{equation*}
\left\Vert \sum\limits_{m}d_{m}f_{m}\right\Vert _{k}=\left\Vert
\sum\limits_{m}d_{m}\left( f_{m}-p_{m}\right) \right\Vert _{k}\leq \sum
\left\vert d_{m}\right\vert \left\Vert \left( f_{m}-p_{m}\right) \right\Vert
_{_{k}} \leq 
\end{equation*}%
\begin{equation*}
\leq \frac{1}{2C_{1}}\sum_{m}\left\vert d_{m}\right\vert \left\Vert
f_{m}\right\Vert _{k}\leq \frac{1}{2}\left\Vert
\sum_{m}d_{m}f_{m}\right\Vert _{_{k}}. 
\end{equation*}%
It follows that $Q$ is one to one and hence an isomorphism. ( see \cite{ED},
p.671). Moreover we have:%
\begin{equation*}
\left\Vert Q\left( \sum_{m}d_{m}f_{m}\right) \right\Vert _{k}=\left\Vert
\sum_{m}d_{m}p_{m}\right\Vert _{k}\leq \sum_{m}\left\vert d_{m}\right\vert
\left\Vert f_{m}\right\Vert _{k}\prec \left\Vert
\sum_{m}d_{m}f_{m}\right\Vert _{k+2A}. 
\end{equation*}

We claim that $Q$ is a tame isomorphism. In order to examine the continuity
estimates of $Q^{-1}$ we shall once again, turn our attention to the
operator $S.$ Consider the Hilbert scale $\left( H_{t}\right) _{t\geq 0}$, \ 
\begin{equation*}
H_{t} = \left\{ \xi =\left( \xi _{m}\right) _{m}:\left\vert \xi \right\vert
_{t} = \left( \sum\limits_{m}\left\vert \xi _{m}\right\vert
^{2}e^{2t\left\vert \sigma \left( m\right) \right\vert }\right) ^{\frac{1}{2}%
}<\infty \right\} ,\text{ \ \ \ \ \ }t\geq 0. 
\end{equation*}%
Fix a number $A^{-}\,\,$close to $A$ yet $A^{-}<A.\,$ The operator $S$ , for
large $k$ induces maps:%
\begin{equation*}
H_{k}\rightarrow O\left( D_{k-A^{-}}\right) 
\end{equation*}%
\begin{equation*}
H_{0}\rightarrow O\left( D_{0}\right) \cap C\left( \overline{D}_{0}\right) . 
\end{equation*}%
In view of Zaharyuta interpolation theorem \cite{Z3}, $S$ extends to be
continuous from $H_{tk}$ into $O\left( D_{t\left( k-A^{-}\right) }\right) $
for each $0\leq t\leq 1.\,\,$Similarly $S^{-1}$, for large $k$ induces maps:%
\begin{equation*}
O\left( \overline{D}_{_{k}}\right) \rightarrow H_{k-A}, 
\end{equation*}%
\begin{equation*}
O\left( \overline{D}_{_{0}}\right) \rightarrow H_{0} 
\end{equation*}%
and again by Zaharyuta interpolation theorem, $S^{-1}$ extends to be
continuous from $O\left( \overline{D}_{tk}\right) $ into $H_{t\left(
k-A\right) }$ for large $k.\,\,$Fix a large $s$ and consider an $\,\,%
\overline{s}<s$ but near $\,\,s.$ Choosing $ks$ as large as needed, we see
that $S$ maps $O\left( \overline{D}_{\overline{s}}\right) $ into $H_{%
\overline{s}}$ continuously and $S^{-1}$ maps $H_{\overline{s}}$ into $%
O\left( D_{_{\overline{s}}}\right) $ continuously. In particular, the
sequence $\left\{ f_{m}\right\}$ forms a basis for the Fr\'{e}chet space $%
O\left( D_{_{s}}\right) .$

For $s\geq k_{1}$ arguing as above, we have$,$

\begin{equation*}
\left\Vert p_{m}\right\Vert _{s}\leq \left\Vert f_{m}\right\Vert
_{s}+\left\Vert p_{m}-f_{m}\right\Vert _{s}\prec \left\Vert f_{m}\right\Vert
_{s}. 
\end{equation*}%
In particular these estimates and the fact that $\left\{ f_{m}\right\}$
forms a basis for $O\left( D_{_{s}}\right) $ allows us to conclude that for
large $s$, the operator $Q$ extends to a continuous operator from $O\left(
D_{_{s}}\right) $ into itself. The argument given above for the
invertibility of $Q$ on $O\left( X\right) $ applies for $Q,$ this time as an
operator from $O\left( D_{_{s}}\right) $ into itself. This in turn will give
us bounds on the continuity estimates of $Q^{-1}.$ Namely for large $k$ we
have 
\begin{equation*}
\left\Vert Q^{-1}\left( f\right) \right\Vert _{k}\prec \left\Vert
f\right\Vert _{k+1}. 
\end{equation*}%
Hence $Q$ is a tame endomorphism of $\left( O\left( X\right) ,{\rho }%
\right).\,\,$Now let $T=Q\circ S^{-1}.$This finishes the proof of the
theorem.
\end{proof}

\begin{remark}
1. The proof given above shows something more, namely that the polynomial
basis found also constitute bases for the Fr\'{e}chet spaces $O\left(
\left\{ z: \rho \left( z\right) <s\right\} \right),\,\,$ for $s$ large.
\end{remark}

\bigskip

\begin{remark}
2. If we only assume that the Stein manifold $X$ is ${S}$-parabolic then in
view of Theorem \ref{thm29} we can choose a Fr\'{e}chet space isomorphism $S 
$ from $O\left( \mathbb{C}^{n}\right) $, $n=dim$ $X,$ onto $O\left( X\right)
.$ The general argument given in the first part of the proof of the above
theorem is valid in this set up so as a corollary of the proof of the
theorem we have:
\end{remark}

\begin{corollary}
Let $\left( X,{\rho }\right) $ be a regular $S-parabolic$ Stein manifold of
dimension n. Then there exits an isomorphism from $O\left( \mathbb{C}%
^{n}\right) $ onto $O\left( X\right) $ that maps polynomials into $\rho-$%
polynomials. In particular $O\left( X\right) $ has a basis consisting of $%
\rho-$polynomials.
\end{corollary}

\end{document}